\documentclass[journal,twoside,web]{ieeecolor}
\usepackage{lcsys}
\usepackage{amsmath,amssymb,amsfonts}
\usepackage{algpseudocode}
\usepackage{algorithm}
\usepackage{graphicx}
\usepackage{graphics}
\usepackage{textcomp}
\usepackage{subcaption}
\usepackage{hyperref}
\usepackage{cleveref}
\usepackage[font=small,skip=0pt]{caption}

\usepackage[backend=bibtex]{biblatex}
\AtNextBibliography{\small}

\def\BibTeX{{\rm B\kern-.05em{\sc i\kern-.025em b}\kern-.08em
    T\kern-.1667em\lower.7ex\hbox{E}\kern-.125emX}}
\newtheorem{theorem}{Theorem}
\newtheorem{lemma}{Lemma}    

\newtheorem{remark}{Remark}

\newtheorem{assumption}{Assumption}

\newcommand{\E}{\mathbb{E}}

\newcommand{\blue}[1]{{\color{black} \noindent #1}}
\newcommand{\blueTwo}[1]{{\color{black} \noindent #1}}

\begin{document}
\title{
Model Predictive Path Integral Control as Preconditioned Gradient Descent
}
\author{Mahyar Fazlyab, Sina Sharifi, Jiarui Wang
\thanks{The authors are with the Department of Electrical and Computer Engineering at Johns Hopkins University, Baltimore, MD 21218, USA.
{\tt\small \{mahyarfazlyab, sshari12, jwang486\}@jhu.edu}}
\thanks{
This work was supported by the National Science Foundation (NSF) under Grant 2515978 and by the U.S. DOE, Office of Science, ASCR program under the Scientific Discovery through Advanced Computing (SciDAC) Institute “LEADS: LEarning-Accelerated Domain Science”.
}
}

\maketitle
\thispagestyle{empty}

\begin{abstract}
\blue{Model Predictive Path Integral (MPPI) control is a widely used sampling-based method for trajectory optimization, yet its convergence properties remain only partially understood. This paper provides a direct convergence analysis using variational optimization. By lifting constrained trajectory optimization to a Kullback-Leibler (KL) regularized problem over decision distributions, we derive a reduced free-energy objective defined over a parametric sampling family. For general parametric families, we derive gradient and Hessian representations of this reduced objective and analyze preconditioned gradient descent on the sampling-distribution parameters. In the fixed-covariance Gaussian case, the classical MPPI update is recovered exactly as a unit-step preconditioned gradient update. We prove descent and stationarity guarantees for the exact expectation-based iteration when the Hessian of the reduced objective is bounded in the metric induced by the preconditioner. 
For the Gaussian family, we further show that the preconditioned Hessian is governed by the covariance of the Gibbs-tilted distribution relative to the covariance of the sampling distribution, yielding a covariance-dependent sufficient condition for the descent of exact unit-step MPPI. Numerical experiments illustrate the theory and the effect of key hyperparameters.
}
\end{abstract}


\section{Introduction}
\IEEEPARstart{M}{odel} Predictive Path Integral (MPPI) control, e.g., \cite{williams2015model, williams2017information}, is a widely used sampling-based method for trajectory optimization in nonlinear and nonconvex settings, owing to its simplicity, parallelizability, and ability to handle nondifferentiable dynamics and costs. In its standard form, MPPI updates the sampling distribution by drawing perturbed control sequences, reweighting them according to their trajectory costs, and shifting the nominal control toward a weighted average of the sampled rollouts. Despite its empirical success in robotics and real-time control, this update is typically introduced through stochastic optimal control or control-as-inference arguments, which do not directly expose its underlying optimization structure. As a result, basic questions such as how MPPI relates to gradient-based methods, when its update is guaranteed to decrease a well-defined objective, and how its hyperparameters influence stability and convergence remain only partially understood~\cite{honda2025model}. These gaps motivate the need for a direct optimization-theoretic interpretation of MPPI.


\subsection{Contributions}
This paper provides a variational, optimization-theoretic \blueTwo{derivation} of MPPI, with the goal of establishing convergence guarantees beyond special cases \cite{yi2024covo}. 
Starting from constrained trajectory optimization, we lift the problem to a KL-regularized distributional formulation and eliminate the auxiliary decision distribution to obtain a reduced negative log-partition, or free-energy objective over a tractable sampling family. We then analyze the convergence of preconditioned gradient descent on the free energy objective. 
Our framework enables three concrete consequences.
\textit{First}, it yields descent and stationarity guarantees, including an \(O(1/K)\) ergodic stationarity rate, for the exact preconditioned-gradient iteration when the Hessian of the reduced objective is bounded in the metric induced by the preconditioner. \textit{Second}, in the fixed-covariance Gaussian family, it recovers classical MPPI exactly as a unit-step preconditioned gradient update and shows that the preconditioned Hessian of the reduced objective 
is governed by the covariance of the Gibbs-tilted distribution relative to the sampling covariance. This leads to an explicit covariance-dependent sufficient condition for descent of exact unit-step MPPI. \textit{Third}, it provides a principled basis for selecting the algorithm hyperparameters, including step size, multiple inner updates, and stopping criteria based on stationarity. \blueTwo{
To the best of our knowledge, our framework provides the first theoretical convergence analysis of MPPI and related sampling algorithms for general nonlinear systems and cost functions, with bounded feasible sets being the only requirement.
}
%
%
Numerical experiments support the theory and illustrate the effect of key hyperparameters on performance.

\subsection{Related Work}

\subsubsection*{Probabilistic Inference Perspective} Inference-based formulations recast control as posterior inference over action sequences conditioned on an optimality variable, leading to updates closely related to MPPI \cite{honda2025model}. This viewpoint has been developed extensively in reinforcement learning and control \cite{levine2018reinforcement}. 
In particular, \cite{okada2020variational} introduced a variational inference MPC framework that recovers several sampling-based optimization methods, including MPPI \cite{williams2016aggressive}, CEM \cite{botev2013cross}, and CMA-ES \cite{hansen2003reducing} as special cases. Our contribution is complementary: rather than deriving MPPI through inference, we show that it can be obtained directly as a preconditioned gradient step on a KL-regularized free-energy objective.



\subsubsection*{Diffusion Perspective}
Another line of work connects MPPI to model-based diffusion \cite{pan2024model, xue2025full, jung2025joint}. 
In \cite{xue2025full}, building on the score estimation result from \cite{pan2024model} that Mscore-estimation result from \cite{pan2024model}, it is shown on a Gaussian-smoothed Gibbs distribution. Although this interpretation explains the mechanism of MPPI, it still does not directly reveal its convergence properties.


\subsubsection*{Optimization Perspective}
\blue{MPPI has also been studied through optimization-based perspectives, particularly distributional mirror descent (MD) and its accelerated variants~\cite{miyashita2018mirror,wagener2019online,okada2018acceleration}. 
\blueTwo{In particular, \cite{okada2018acceleration} hypothesized a $\mathcal{O}(1/k)$ convergence guarantee.}
%
Wagener et al.~\cite{wagener2019online} considered utility-transformed trajectory objectives and showed that
\blueTwo{MD on}
the exponential-utility yields classical MPPI under a fixed-covariance Gaussian family with unit step size.
In contrast, \blueTwo{in our framework, the free-energy objective arises naturally from the} KL-regularized variational formulation of the original constrained trajectory optimization problem.
\blueTwo{More importantly, none of the aforementioned approaches provide convergence guarantees.}
}

\subsubsection*{Theoretical Analysis of MPPI}
Motivated by the empirical success of MPPI, several recent works have begun to study its theoretical properties. In particular, CoVO-MPC\cite{yi2024covo} analyzes the convergence behavior of MPPI using contraction theory, proving at least linear convergence for (time-varying) LQR. However, the same contraction result cannot be extended to general nonlinear settings. Separately, \cite{homburger2025optimality} studies the optimality and suboptimality of MPPI in stochastic and deterministic settings, with an emphasis on deterministic MPPI and its approximation error. Our analysis is complementary to these works: we analyze the convergence for general nonlinear systems and cost, with bounded feasible set being the main  requirement. 



\subsection{Notation}

For a symmetric matrix \(A\), \(A\succeq0\) and \(A\succ0\) denote positive semidefiniteness and positive definiteness. The identity matrix is \(I\), and \(\lambda_{\min}(A)\), \(\lambda_{\max}(A)\) denote the extreme eigenvalues of \(A\). We use \(\|\cdot\|\) for both the Euclidean and spectral norms. For \(P\succ0\), let \(\|x\|_P^2=x^\top Px\). For a density \(\pi\), \(\mathbb E_\pi[\cdot]\), \(\mathrm{Cov}_\pi(\cdot)\), and \(\mathrm{supp}(\pi)\) denote expectation, covariance, and support. We write \(\pi(u)=\mathcal N(u;\mu,\Sigma)\) for a Gaussian density and \(\mathrm{KL}(\rho\|\pi)\) for the Kullback--Leibler divergence. The notation \(\rho\ll\pi\) means that \(\rho\) is absolutely continuous with respect to \(\pi\). For a differentiable function \(F\), \(\nabla F\) and \(\nabla^2F\) denote its gradient and Hessian. For a set \(C\), \(\mathbf 1_C\) denotes its indicator function.

\section{Variational Formulation}

\subsection{Trajectory Optimization as Constrained Minimization}
We consider finite-horizon trajectory optimization over an open-loop control sequence $u := (u_0,u_1,\dots,u_{T-1}) \in \mathbb{R}^{dT}$ applied to a dynamical system
\[
x_{t+1} = F(x_t,u_t),
\]
possibly nonlinear and nonsmooth, from a given initial condition \(x_0\). Let \(f_0(u)\) denote the trajectory objective (e.g., cumulative stage costs and a terminal cost), and \(C\subset\mathbb R^{dT}\) denote the set of feasible control sequences, encoding constraints such as obstacle avoidance, state bounds, or input limits. Throughout, we assume that \(C\) is nonempty and compact and that \(f_0\) is continuous. The resulting trajectory optimization problem is
\begin{align}\label{eq:det_to}
    \min_{u\in C} f_0(u).
\end{align}

%

\subsection{KL-Regularized Distributional Formulation}
\blue{Following the framework of variational optimization, e.g., ~\cite{staines2013optimization}, we first lift the pointwise trajectory optimization problem \eqref{eq:det_to} to an optimization problem over probability distributions on the open-loop control sequence \(u\):
\[
\min_{\rho}
\mathbb E_\rho[f_0(u)] \quad \mathrm{s.t.} \ \mathrm{supp}(\rho)\subseteq C
\]
Here, \(\rho\) denotes a \emph{decision distribution} over \(u\). This unregularized lifted problem is equivalent to the original pointwise problem and collapses to a Dirac measure at a minimizer of \(f_0\) over \(C\). To obtain a nondegenerate distributional formulation, we introduce a \emph{base}, or sampling, distribution \(\pi\) over the same control-sequence space and regularize \(\rho\) relative to \(\pi\) using the KL divergence. We require \(\rho\ll\pi\); otherwise, \(\mathrm{KL}(\rho\|\pi)=+\infty\). For a regularization parameter \(\tau>0\), consider
\begin{equation}
\label{eq:dist_primal}
\min_{\rho} \ 
\mathbb{E}_{\rho}[f_0(u)] 
+ 
\tau \, \mathrm{KL}(\rho\|\pi)
\quad 
\text{s.t.} 
\quad 
\mathrm{supp}(\rho) \subseteq C .
\end{equation}
The support constraint enforces the hard feasibility of the decision distribution. Problem~\eqref{eq:dist_primal} trades off low expected trajectory cost under \(\rho\) with proximity to the sampling distribution \(\pi\). As \(\tau\to 0\), the regularization vanishes and optimal solutions concentrate on the optimal set
\(
U^\star=\arg\min_{u\in C} f_0(u).
\)
}

\subsection{Optimizing the Base Distribution}
For any fixed base distribution $\pi$, \eqref{eq:dist_primal} provides an upper bound on the optimal value of the original constrained problem:
\[
\min_{\rho}
\Big(
\E_\rho[f_0(u)] + \tau \mathrm{KL}(\rho\|\pi)
\Big)
\ge
\min_{\rho}\E_\rho[f_0(u)]
=
\min_{u\in C} f_0(u),
\]
where both minimizations are taken over distributions \(\rho\) supported on \(C\), and the equality follows by choosing \(\rho\) as a Dirac measure at any minimizer of \(f_0\) over \(C\). \blue{Note that the upper bound is a function of $\pi$.} Therefore, we can optimize over \(\pi\) to seek the tightest such upper bound. However, if we optimize over \(\pi\) without restriction, the pair $(\rho,\pi)$ may collapse (e.g., $\pi=\rho$), undermining stability and exploration. We therefore restrict $\pi$ to a tractable family $\Pi$ (e.g., Gaussians with bounded covariance), and consider
\begin{equation}
\label{eq:dist_joint}
\min_{\pi \in \Pi}\min_{\rho} \ \mathbb{E}_{\rho}[f_0(u)] + \tau \, \mathrm{KL}(\rho\|\pi)
\quad \text{s.t.} \quad \mathrm{supp}(\rho) \subseteq C.
\end{equation}
For a fixed $\pi\in\Pi$, the minimizer over $\rho$ in \eqref{eq:dist_joint} is given by the truncated Gibbs tilt
\begin{align}
\label{eq:rho_star}
\rho^\star_{\pi}(u)
&=T(\pi)(u):=
\frac{\pi(u)\exp\!\big(-f_0(u)/\tau\big)\,\mathbf{1}_{C}(u)}
{Z(\pi)},
\end{align}
where $Z(\pi)$ is the normalizing constant
\begin{align}\label{eq:partion_function}
Z(\pi)
&:=
\int_{C} \pi(v)\exp\!\big(-f_0(v)/\tau\big)\,dv.
\end{align}
See the full derivation in the longer version.
\blue{Since $\pi$ is positive on $C$, we have $Z(\pi)>0$}.
Thus, \(\rho^\star_\pi\) is obtained by reweighting the base distribution \(\pi\) according to trajectory cost and feasibility: lower-cost feasible control sequences receive larger probability mass, whereas infeasible sequences receive zero mass.
Substituting \eqref{eq:rho_star} into the inner objective yields
\begin{align}
&
\mathbb{E}_{\rho^\star_\pi}[f_0(u)]
+
\tau\,\mathrm{KL}(\rho^\star_\pi\|\pi)
\nonumber\\
&\quad=
\mathbb{E}_{\rho^\star_\pi}[f_0(u)]
+
\tau\,
\mathbb{E}_{\rho^\star_\pi}
\left[
-\frac{f_0(u)}{\tau}
-
\log Z(\pi)
\right]
\nonumber\\
&\quad=
-\tau\log Z(\pi).
\end{align}
Therefore, the joint optimization problem \eqref{eq:dist_joint} reduces to the optimization problem over the negative log-partition or free-energy objective,
\begin{equation}
\label{eq:reduced}
\min_{\pi\in\Pi} -\tau \log Z(\pi).
\end{equation}
This objective is precisely the \(\pi\)-dependent upper bound obtained after eliminating the auxiliary distribution \(\rho\).





\section{Optimization over a Parametric Sampling Family}
\label{sec:parametric_family}

In this section, we specialize the reduced problem \eqref{eq:reduced} to a parametric family of sampling distributions $\Pi := \{\pi_\theta : \theta\in\Theta\}$, where \(\Theta\subseteq\mathbb{R}^p\) is the parameter space. For each \(\theta\in\Theta\), the corresponding optimal decision distribution is
\begin{align}
\label{eq:rho_theta}
\rho_\theta(u) := T(\pi_\theta)(u).
\end{align}
Accordingly, the reduced problem becomes
\begin{align}
\label{eq:reduced_theta}
\min_{\theta\in\Theta}\; F(\theta) \! := \! -\tau \log Z(\theta) \! = \! -\tau
\log \int_C \pi_\theta(u)e^{-\tfrac{f_0(u)}{\tau}}\,du.
\end{align}
We make the following assumption, under which the reduced objective \(F\) becomes twice differentiable.

\begin{assumption}
\label{ass:regularity_parametric}
The family \(\{\pi_\theta\}_{\theta\in\Theta}\) is strictly positive on \(C\), twice continuously differentiable in \(\theta\), and such that differentiation under the integral sign is valid for \(Z(\theta)\) up to second order.
\end{assumption}

\subsection{Preconditioned Gradient Descent}
\label{sec:pgd_mppi}

In contrast to the original constrained trajectory optimization problem, the reduced problem \eqref{eq:reduced_theta} is differentiable in the distribution parameters and is therefore amenable to gradient-based optimization. The following result provides expressions for the gradient and Hessian of \(F(\theta)\) that will be useful for algorithm design and convergence analysis.

\begin{lemma}[Gradient and Hessian Representations]
\label{lem:gradient-representations}
Under Assumption~\ref{ass:regularity_parametric},
\begin{align}
\nabla F(\theta)
&=
-\tau \mathbb{E}_{\rho_\theta}\!\left[\nabla_\theta \log \pi_\theta(u)\right]
\label{eq:grad_F_theta} \\
&=
-\tau \frac{\mathbb{E}_{\pi_\theta}\!\left[w(u)\nabla_\theta \log \pi_\theta(u)\right]}
{\mathbb{E}_{\pi_\theta}\!\left[w(u)\right]},
\label{eq:grad_F_theta_2}
\end{align}
where
\begin{align}\label{eq:weight_sample}
w(u):=\exp\!\big(-f_0(u)/\tau\big)\mathbf{1}_C(u).    
\end{align}
In addition,
\begin{align}
\nabla^2 F(\theta)
&=
-\tau
\Big(
\mathbb{E}_{\rho_\theta}\!\left[\nabla_\theta^2 \log \pi_\theta(u)\right]
+
\operatorname{Cov}_{\rho_\theta}\!\left(\nabla_\theta \log \pi_\theta(u)\right)
\Big).
\label{eq:hessian_general}
\end{align}
\end{lemma}
\begin{proof}
See the proof in the longer version.
\end{proof}
The gradient representations in Lemma~\ref{lem:gradient-representations} naturally motivate a preconditioned gradient method for minimizing the reduced objective $F(\theta)$.
 Given a symmetric positive definite preconditioner $P \succ 0$ and a step size $\eta>0$, the exact preconditioned gradient descent is
\begin{align}
\label{eq:pgd_exact}
\theta_{k+1}
&=
\theta_k-\eta P \nabla F(\theta_k) \\
&=
\theta_k+\eta \tau P\,
\mathbb{E}_{\rho_{\theta_k}}
\!\left[\nabla_\theta \log \pi_{\theta_k}(u)\right].
\nonumber
\end{align} 
\begin{algorithm}[t]
\caption{Multi-step MPPI}
\label{alg:GD-MPPI}
\begin{algorithmic}[1]
    \Require Initial parameter $\theta_0$, number of samples $N$, threshold $\varepsilon$
    \State Set $k = 1.$
    \While{\blueTwo{$\|\nabla F(\theta^k)\|_{P} > \varepsilon$}}
        \State Sample $u^{(1)}, ..., u^{(N)} \overset{\text{i.i.d.}}{\sim} \pi_{\theta_{k-1}}(u)$
        \State Approximate $\nabla_\theta F(\theta_{k})$ according to \eqref{eq:pgd_weights}.
        \State Update $\theta_{k}$ according to \eqref{eq:pgd_sampled}.
        \State $k \gets k + 1$
    \EndWhile
\end{algorithmic}
\end{algorithm}
In practice, the expectation with respect to $\rho_{\theta_k}$ is generally intractable. Using \eqref{eq:grad_F_theta_2}, we can sample from   $\pi_{\theta_k}$ instead. Specifically, we draw samples $u^{(j)}\sim \pi_{\theta_k}, \  j=1,\dots,N$, 
and define the self-normalized importance weights
\begin{align}
\label{eq:pgd_weights}
\bar w_j:=\frac{w(u^{(j)})}{\sum_{r=1}^N w(u^{(r)})} \quad j=1,\cdots,N.
\end{align}
This yields the self-normalized Monte Carlo estimator of \eqref{eq:pgd_exact},
\begin{align}
\label{eq:pgd_sampled}
\theta_{k+1}
&=
\theta_k+\eta \tau P
\sum_{j=1}^N \bar w_j\,
\nabla_\theta \log \pi_{\theta_k}(u^{(j)}).
\end{align}
This update has the standard structure of a weighted sample average used in policy search and sampling-based control methods, and recovers MPPI as a special case under a suitable Gaussian parameterization.

\subsection{Convergence Analysis}
\label{sec:convergence}

We now analyze the exact preconditioned gradient iteration \eqref{eq:pgd_exact} with a constant step size \(\eta>0\), and derive conditions under which it yields descent and convergence of the reduced objective. Since the iteration is preconditioned by \(P\), the relevant notion of smoothness is naturally expressed in the metric induced by \(P\).

\begin{assumption}
\label{ass:hessian_bound}
For the chosen positive definite matrix \(P\succ 0\), there exists a constant \(L_P>0\) such that
\[
\sup_{\theta\in\Theta}
\left\|
P^{1/2}\nabla^2 F(\theta)P^{1/2}
\right\|
\le L_P.
\]
\end{assumption}

The following lemma is an immediate consequence of Assumption~\ref{ass:hessian_bound}.

\begin{lemma}
\label{lemma:LP_smooth_from_hessian}
Under Assumption~\ref{ass:hessian_bound}, for all \(\theta,\theta+\Delta\theta\in\Theta\),
\begin{equation}
\label{eq:LP_smooth}
F(\theta+\Delta\theta)
\le
F(\theta)
+
\nabla F(\theta)^\top \Delta\theta
+
\frac{L_P}{2}\Delta\theta^\top P^{-1}\Delta\theta.
\end{equation}
\end{lemma}
\begin{proof}
    See the proof in the longer version.
\end{proof}

\begin{theorem}
\label{thm:pgd_general}
Suppose Assumption~\ref{ass:hessian_bound} holds, and let \(\{\theta^k\}\) be generated by \eqref{eq:pgd_exact} with a constant step size \(\eta>0\). If
\begin{equation}
\label{eq:eta_condition_general}
0<\eta<\frac{2}{L_P},
\end{equation}
then the following hold:
\begin{enumerate}
\item \textbf{Descent:} for every \(k\),
\begin{equation}
\label{eq:descent_general}
F(\theta^{k+1})
\le
F(\theta^k)
-
\eta\Big(1-\frac{\eta L_P}{2}\Big)
\|\nabla F(\theta^k)\|_{P}^2.
\end{equation}

\item \textbf{Summability of preconditioned gradients:}
\begin{equation}
\label{eq:summability_general}
\sum_{k=0}^{\infty}\|\nabla F(\theta^k)\|_{P}^2
\le
\frac{F(\theta^0)-\inf_{\theta\in\Theta}F(\theta)}
{\eta\big(1-\frac{\eta L_P}{2}\big)}.
\end{equation}

\item \textbf{Stationarity:} for every \(K\ge 1\),
\begin{equation}
\label{eq:min_stationarity_general}
\min_{0\le j\le K-1}\|\nabla F(\theta^j)\|_P^2
\le
\frac{F(\theta^0)-\inf_{\theta\in\Theta}F(\theta)}
{K\,\eta\big(1-\frac{\eta L_P}{2}\big)}.
\end{equation}
In particular, $\lim_{k \to \infty} \|\nabla F(\theta^k)\|_P=0$. 
\end{enumerate}
\end{theorem}

\begin{proof}
Applying the smoothness bound \eqref{eq:LP_smooth} with $\Delta \theta=-\eta P\nabla F (\theta^k)$, we obtain \eqref{eq:descent_general} after simplification.  Since \(0<\eta<2/L_P\), the coefficient is positive, so \(F(\theta^k)\) is nonincreasing. Summing \eqref{eq:descent_general} from \(k=0\) to \(K-1\) and using $\inf_{\theta\in\Theta}F(\theta) \leq F(\theta^K)$
\[
\eta\Big(1-\frac{\eta L_P}{2}\Big)
\sum_{k=0}^{K-1}\|\nabla F(\theta^k)\|_P^2
\le
F(\theta^0)-\inf_{\theta\in\Theta}F(\theta).
\]
Letting \(K\to\infty\) gives \eqref{eq:summability_general}. Dividing by \(K\) gives \eqref{eq:min_stationarity_general}, and summability implies \(\|\nabla F(\theta^k)\|_P\to 0\).
\end{proof}

\blue{
Theorem~\ref{thm:pgd_general} shows that the exact multi-step iteration \eqref{eq:pgd_exact} is a descent method for the free-energy objective \(F(\theta)\): for a sufficiently small step size, \(F(\theta^k)\) decreases monotonically and the iterates converge toward stationarity. This result also suggests a natural stopping criterion based on the preconditioned gradient norm \(\|\nabla F(\theta^k)\|_{P}\). See Algorithm~\ref{alg:GD-MPPI} for a summary of the method.}

\section{Optimization over Gaussian Family with Fixed Covariance}
\label{sec:gaussian_fixed_cov}

We now specialize the preceding results to the fixed-covariance Gaussian family
\begin{align}
\label{eq:gaussian:family}
\pi_\mu(u)=\mathcal{N}(u;\mu,\Sigma),
\qquad
\Sigma\succ 0,
\end{align}
where the mean \(\mu\in\mathbb{R}^m\) is the optimization variable. In this case, the score and log-Hessian are given by
\[
\nabla_\mu \log \pi_\mu(u)=\Sigma^{-1}(u-\mu),
\qquad
\nabla_\mu^2 \log \pi_\mu(u)=-\Sigma^{-1}.
\]
Substituting these expressions into Lemma~\ref{lem:gradient-representations} yields
\begin{align}
\nabla F(\mu)
&=
-\tau \Sigma^{-1}\big(\E_{\rho_\mu}[u]-\mu\big), \  \rho_{\mu}(u)=T(\pi_{\mu})(u).
\label{eq:grad_gaussian_conv}
\end{align}
Using \eqref{eq:grad_gaussian_conv}, the exact preconditioned gradient step \eqref{eq:pgd_exact} becomes
\begin{align}
\label{eq:gaussian_pgd_exact}
\mu_{k+1}
=
\mu_k
+\eta \tau P \Sigma^{-1}
\bigl(\E_{\rho_{\mu_k}}[u]-\mu_k\bigr).
\end{align}
Using the ratio-of-expectations representation in \eqref{eq:grad_F_theta_2}, the expectation $\mathbb{E}_{\rho_{\mu_k}}[u]$ can be approximated by self-normalized importance sampling. Accordingly, if \(u^{(j)} \sim \mathcal{N}(\mu_k,\Sigma)\) and the normalized weights \(\bar w_j\) are defined as in \eqref{eq:pgd_weights}, a Monte Carlo implementation of \eqref{eq:gaussian_pgd_exact} is
\begin{align}
\label{eq:gaussian_pgd_sampled}
\mu_{k+1}
=
\mu_k
+\eta \tau P \Sigma^{-1}
\left(
\sum_{j=1}^N \bar w_j u^{(j)}-\mu_k
\right).
\end{align}
In particular, by choosing $P=\frac{1}{\tau}\Sigma, \ 
\eta=1$, the exact preconditioned gradient update \eqref{eq:gaussian_pgd_exact} reduces to
\begin{align}
\mu_{k+1}
&=
\E_{\rho_{\mu_k}}[u] =
\frac{\E_{\pi_{\mu_k}}[w(u)\,u]}{\E_{\pi_{\mu_k}}[w(u)]},
\label{eq:mppi_exact_expectation}
\end{align}
where \(w(u)=\exp(-f_0(u)/\tau)\mathbf 1_C(u)\), and the last equality follows from the ratio-of-expectations in Lemma~\ref{lem:gradient-representations}. Correspondingly, the Monte Carlo implementation becomes
\begin{align}
\label{eq:mppi_exact_update}
\mu_{k+1}
=
\sum_{j=1}^N \bar w_j u^{(j)},
\qquad
u^{(j)}\sim\mathcal N(\mu_k,\Sigma),
\end{align}
which is precisely the classical MPPI update.

\subsection{Convergence analysis}
\label{subsec:convergence}
We now analyze the exact Gaussian update \eqref{eq:gaussian_pgd_exact} through the lens of preconditioned gradient descent. Although \eqref{eq:gaussian_pgd_exact} is well defined for any positive definite preconditioner \(P\), the choice
\[
P=\frac{1}{\tau}\Sigma
\]
is especially natural for two reasons. First, when \(\eta=1\), this choice exactly recovers the classical MPPI update, as shown in the previous subsection. Therefore, convergence guarantees established under \(P=\Sigma/\tau\) immediately apply to MPPI, as well as to its relaxed version with arbitrary step size \(\eta>0\). 
Second, this preconditioner is intrinsic to the geometry of the fixed-covariance Gaussian family. Indeed, by Lemma~\ref{lem:gradient-representations},
\[
\nabla^2 F(\mu)
=
\tau \Sigma^{-1}
-
\tau \Sigma^{-1}\mathrm{Cov}_{\rho_\mu}(u)\Sigma^{-1}.
\]
Hence, with \(P=\Sigma/\tau\),
\begin{align}
\label{eq:metric_hessian_sigma}
P^{1/2}\nabla^2 F(\mu)P^{1/2}
&=
\left(\tfrac{\Sigma}{\tau}\right)^{1/2}
\nabla^2 F(\mu)
\left(\tfrac{\Sigma}{\tau}\right)^{1/2} \\
&=
I-\Sigma^{-1/2}\mathrm{Cov}_{\rho_\mu}(u)\Sigma^{-1/2}.
\nonumber
\end{align}
Thus, in the metric induced by \(P=\Sigma/\tau\), the curvature of the reduced objective is determined entirely by the covariance of the tilted distribution \(\rho_\mu\) relative to the sampling covariance \(\Sigma\). In particular, the explicit dependence on the temperature \(\tau\) disappears after preconditioning. This makes \(P=\Sigma/\tau\) the natural scaling for the convergence analysis.

Accordingly, throughout this subsection we specialize \eqref{eq:gaussian_pgd_exact} to the update
\begin{align}
\label{eq:relaxed_mppi_exact}
\mu_{k+1}
&=(1-\eta)\mu_k+\eta\,\mathbb E_{\rho_{\mu_k}}[u],
\end{align}
which we refer to as the \emph{exact relaxed MPPI update}.
To state the convergence result, define
\begin{align}
\label{eq:L_sigma_def}
L_\Sigma
:=
\sup_{\mu\in\Theta}
\left\|
I-\Sigma^{-1/2}\mathrm{Cov}_{\rho_\mu}(u)\Sigma^{-1/2}
\right\|.
\end{align}
By \eqref{eq:metric_hessian_sigma}, \(L_\Sigma\) is the operator-norm bound on the Hessian of \(F\) in the metric induced by \(\Sigma/\tau\), and hence the corresponding smoothness constant in that metric. In the next theorem, we state the convergence result for \eqref{eq:relaxed_mppi_exact}.

\begin{figure*}[t]
    \centering
    \includegraphics[width=0.32\linewidth]{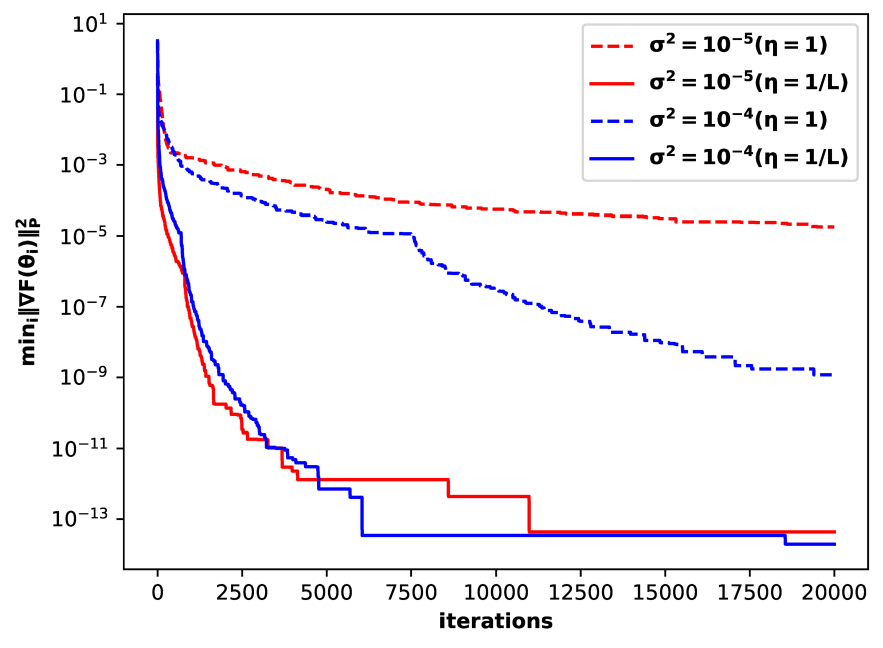}
    \includegraphics[width=0.32\linewidth]{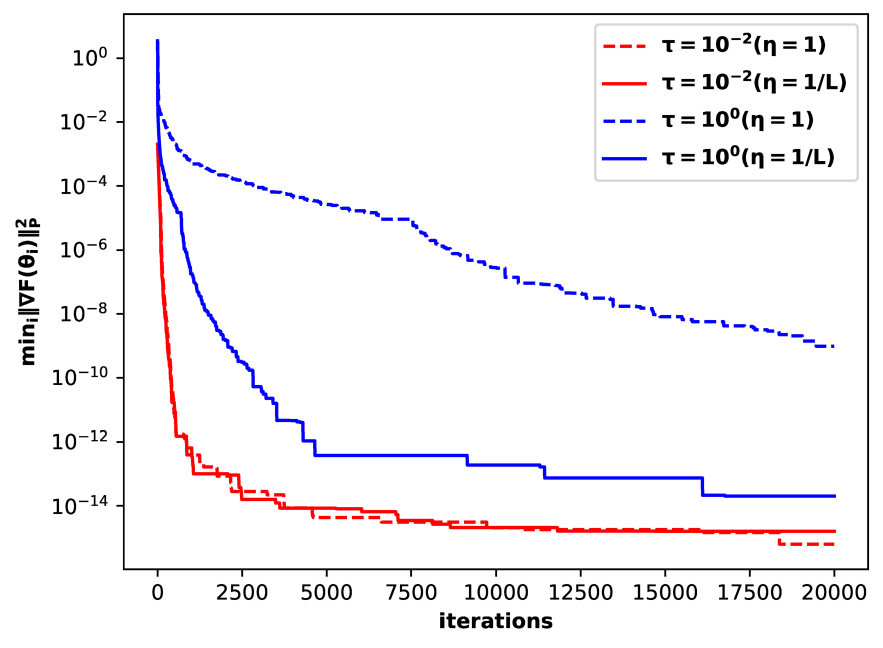}
    \includegraphics[width=0.315\linewidth]{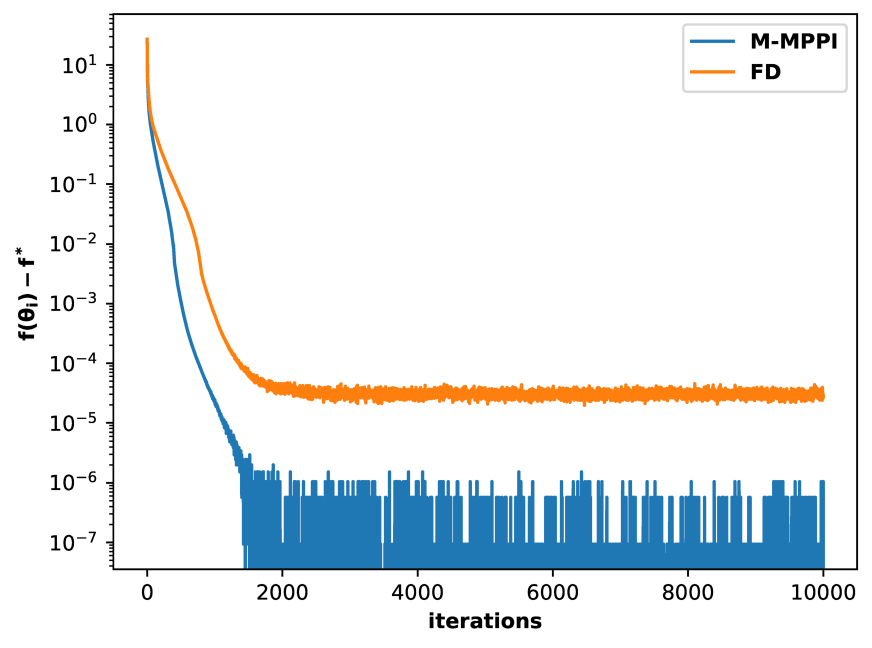}
    \caption{Ablation study of the parameters $\Sigma$ (left) and $\tau$ (middle), and comparison with finite differences (right) on the LQR benchmark.}
    \label{fig:LQR}
\end{figure*}

\begin{theorem}
\label{thm:gaussian_stepsize_bounded}
Consider the fixed-covariance Gaussian family \eqref{eq:gaussian:family}, and the exact preconditioned gradient update \eqref{eq:gaussian_pgd_exact}. 
Assume that the feasible set \(C\subset \mathbb R^m\) is bounded, with diameter \blue{ $D_{\Sigma^{-1}}:=\sup_{u,v\in C}\|\Sigma^{-1/2}(u-v)\|$}. Then the metric smoothness constant \eqref{eq:L_sigma_def} satisfies
\begin{align}
L_\Sigma \le
\max\left\{1,\frac{\blue{D_{\Sigma^{-1}}^2}}{4}-1\right\}.
\label{eq:LP_gaussian_bounded}
\end{align}
Consequently, the exact relaxed MPPI update \eqref{eq:relaxed_mppi_exact} satisfies the descent and convergence conclusions of Theorem~\ref{thm:pgd_general} whenever
\begin{align}\label{eq:convergence_condition}
    0<\eta<\frac{2}{L_\Sigma}.
\end{align}
\end{theorem}
\begin{proof}
    See the proof in the longer version.
\end{proof}
\subsubsection*{Implication for MPPI with unit step size}
\label{rem:mppi_unit_stepsize}
The exact MPPI iteration is recovered by setting \(\eta=1\) in \eqref{eq:relaxed_mppi_exact}. Hence, convergence of the exact MPPI iteration follows from Theorem~\ref{thm:gaussian_stepsize_bounded} whenever the unit step size satisfies the admissibility condition for $0<1<\frac{2}{L_\Sigma}$, which is equivalent to $L_\Sigma<2$.
Using the bound in \eqref{eq:LP_gaussian_bounded}, a sufficient condition is therefore $D^2_{\Sigma^{-1}} < 12$. 
\blue{Since $D^2_{\Sigma^{-1}} \le
\frac{D^2}{\lambda_{\min}(\Sigma)}$, where $D$ is the Euclidean diameter of $C$, this condition is guaranteed when
\(
\lambda_{\min}(\Sigma) \geq \frac{D^2}{12}.
\)
}

Thus, if the covariance matrix is sufficiently large, then the exact MPPI iteration with \(\eta=1\) satisfies the descent and convergence guarantees of Theorem~\ref{thm:pgd_general}. In particular, this gives a simple design rule: the exploration covariance must not be too small relative to the diameter of the feasible set. Equivalently, overly concentrated sampling distributions can destroy the global descent guarantee, whereas sufficiently diffuse sampling is enough to ensure it.

\blue{
\begin{remark}
Theorems~\ref{thm:pgd_general}--\ref{thm:gaussian_stepsize_bounded} analyze the exact expectation-based iteration. The sampled update
\eqref{eq:pgd_sampled} instead uses a self-normalized importance-sampling
estimator, which is generally biased. To see how this affects the descent guarantee, write
\[
\widehat{\nabla F}(\theta_k)
=
\nabla F(\theta_k)+b_k+\xi_k,
\]
where
$
b_k
:=
\mathbb E[\widehat{\nabla F}(\theta_k)\mid\theta_k]
-
\nabla F(\theta_k)
$
is the bias of the self-normalized estimator, and
$\mathbb E[\xi_k\mid\theta_k]=0$ captures its zero-mean random fluctuation. 
Applying the same smoothness argument as in Theorem~\ref{thm:pgd_general},
\[
\begin{aligned}
\mathbb E[F(\theta_{k+1})\mid\theta_k]
&\le
F(\theta_k)
-
\eta\left(\frac34-L_P\eta\right)
\|\nabla F(\theta_k)\|_P^2 \\
&+
\eta(1+L_P\eta)\|b_k\|_P^2
+
\frac{L_P\eta^2}{2}
\mathbb E[\|\xi_k\|_P^2\mid\theta_k].
\end{aligned}
\]
Thus, the exact descent guarantee is preserved up to two terms controlled by the finite-sample gradient estimation error. A complete non-asymptotic analysis of the bias and variance of the self-normalized estimator is an
important direction for future work.
\end{remark}
}

\section{Numerical Analysis}
\subsection{Linear Quadratic Regulator (LQR)}
We consider a finite-horizon LQR trajectory optimization problem with double-integrator dynamics
\begin{align}
x_{t+1} = \begin{bmatrix}
1 & 1 \\
0 & 1
\end{bmatrix} x_t + \begin{bmatrix}
0.5 \\
1
\end{bmatrix} u_t. \notag
\end{align}
\blue{We define the cost to be $J(u) = \sum_{t = 1}^T \|x_t\|^2 + \|u_t\|^2,$
where horizon $T=10$, $u :=  (u_0, \cdots, u_{T-1}) \in \mathbb{R}^{10}$ is the stacked control vector, $x_0 = (2.5, 0)$.
}
Rolling out the trajectories yields the quadratic program
\begin{align}
    \min_{u \in \mathcal{C}} \tfrac{1}{2} u^\top Q u + c^\top u, \notag
\end{align}
\blue{where \(Q \succeq 0\), and \(c\) can be computed accordingly. The constraint set \(\mathcal{C}\) enforces both the control bounds \(|u| \leq 1\) and the state constraints \(x \in [-5,5] \times [-1,1]\). 
A detailed derivation of the resulting QP is provided in the longer version.
}
We set a budget of $N=1000$ samples per iteration.
\Cref{fig:LQR} shows the results for different choices of the parameters.
\blue{
\Cref{fig:LQR} illustrates the convergence behavior predicted by \eqref{eq:convergence_condition}. 
%
In the first two figures, we fix one parameter among $\tau$ and $\Sigma$, and compare the choices of MPPI $\eta = 1$ (dashed lines), and suggested by our theories ($\eta = 1/L_\Sigma$) (solid lines). 
In the left plot, we fix $\tau = 1$ and compare two choices of $\Sigma = \sigma^2 I$, and in the middle plot, we fix $\Sigma=10^{-4}I$, and compare two choices of $\tau$.
Since the LQR objective is quadratic, $L_\Sigma$ can be computed explicitly. 
When the Lipschitz constant is small (e.g., $L_\Sigma = 0.1$), the choice $\eta = 1$ becomes conservative, and a larger step size leads to a faster convergence rate.
%
%
We also compare Multi-step MPPI (M-MPPI) with finite differences (FD) in the right figure, where our method outperforms FD. More details regarding the setup can be found in
the longer version.
}



\subsection{Dubins Car}
We then consider a trajectory optimization task in a cluttered environment, where a Dubins car must reach a given destination.
At each time, the optimization problem is formulated as
\begin{align}
    \min_{u \in \mathcal{C}} \sum_{t=1}^T 
    \|x_t - x_d\|_Q^2 + \|u_t\|_R^2, \notag
\end{align}
where $T=20$, $Q = \mathrm{diag}(1, 1, 0.01)$, $R=0.001$, $x_0 = (0, 0, \pi/2)$, and $x_d = (6, 6, 0)$.
The system dynamics $x_t = [p^x_{t}, p^y_{t},
        \mathtt{\theta}_{t}]^\top$ are
\begin{align*}
    x_{t+1} =  x_{t} + [v \cos(\mathtt{\theta}_t), v \sin(\mathtt{\theta}_t),w_t]^\top \Delta t,
\end{align*}
where $v = 4$ is the constant velocity and the control $w_t \in [-\frac{3}{2}\pi, \frac{3}{2}\pi]$, and we set $N=1024$.
\Cref{fig:dubins} shows the trajectory chosen by the algorithm with $K=1$ (MPPI) and $K=10$. Since MPPI does not iterate until convergence, it selects a suboptimal path. More details on this setup and comparison with Log-MPPI \cite{mohamed2022autonomous} are in \Cref{tab:dubins}, where we show that increasing $K$ improves the average cost, at the expense of runtime. The reported results are averaged over 3 seeds.

\begin{figure}[t]
    \centering
    \includegraphics[width=0.49\linewidth]{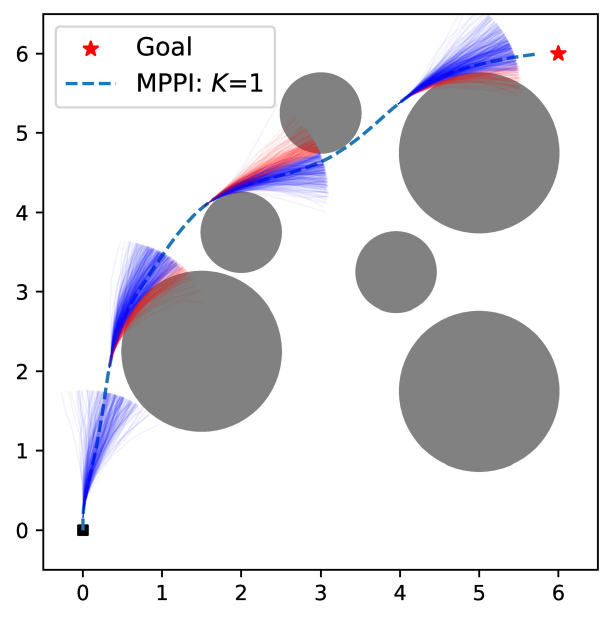}
    \includegraphics[width=0.49\linewidth]{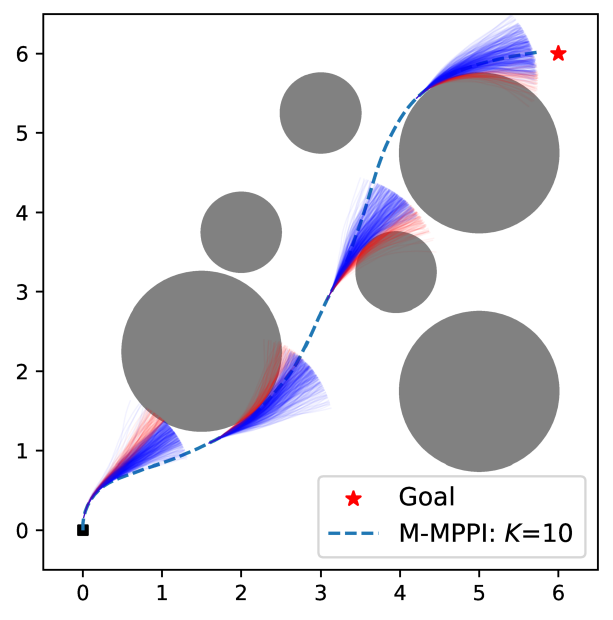}
    \caption{Comparison of the trajectories chosen by MPPI and 10-step MPPI on the Dubins car benchmark in a cluttered environment. \blue{The blue and red lines denote the safe and unsafe trajectories, respectively.}}
    \label{fig:dubins}
\end{figure}

\begin{table}[t]
\centering
\begin{tabular}{l c c c}
\hline
 & Runtime & Sample acceptance & Average  \\
 & (s) & rate \% & cost\\
\hline
MPPI ($K=1$) & \textbf{16.8} & 0.75 & 26.14\\
Log-MPPI & 17.1 & 0.75 & 24.3\\
M-MPPI ($K=5$) & 34.4 & 0.79 & 23.98\\
M-MPPI ($K=10$) & 47.8 & \textbf{0.80} & \textbf{23.85}\\
\hline
\end{tabular}
\caption{Comparison of runtime, sample acceptance rate $\%$, and average cost of the chosen trajectory for various methods. }
\label{tab:dubins}
\end{table}

\section{Conclusion}
In this paper, we showed that MPPI admits a direct variational and optimization-theoretic interpretation. By lifting constrained trajectory optimization to a KL-regularized problem over distributions, we obtained a free-energy objective whose optimization over a parametric sampling family yields a preconditioned gradient method. In the Gaussian fixed-covariance setting, this recovers classical MPPI exactly and leads to explicit descent and stationarity guarantees, as well as a simple covariance-dependent design rule for unit-step MPPI. These results help demystify MPPI from an optimization viewpoint and open the door to principled extensions of sampling-based control methods.
\blue{Our analysis focuses on the exact expectation-based iteration; understanding the full finite-sample and receding-horizon closed-loop behavior remains an important direction for future work.}

\printbibliography

\if

\clearpage
\appendix

\subsection{Derivation of the Truncated Gibbs Distribution}
\label{appendix:gibbs_derivation}
Define
\begin{align}
\label{eq:rho_star_appendix}
\rho^\star_\pi(u)
:=
\frac{
\pi(u)\exp\!\big(-f_0(u)/\tau\big)\mathbf 1_C(u)
}{
Z(\pi)
}.
\end{align}
We show that \(\rho^\star_\pi\) is the unique minimizer of
\eqref{eq:dist_primal}. For any feasible distribution \(\rho\),
\[
\log \rho^\star_\pi(u)
=
\log \pi(u)
-
\frac{f_0(u)}{\tau}
-
\log Z(\pi),
\qquad u\in C.
\]
Hence,
\begin{align}
\mathrm{KL}(\rho\|\rho^\star_\pi)
&=
\int_C
\rho(u)
\log\frac{\rho(u)}{\rho^\star_\pi(u)}
\,du \notag\\
&=
\int_C
\rho(u)
\left[
\log\frac{\rho(u)}{\pi(u)}
+
\frac{f_0(u)}{\tau}
+
\log Z(\pi)
\right]du \notag\\
&=
\mathrm{KL}(\rho\|\pi)
+
\frac{1}{\tau}\mathbb E_{\rho}[f_0(u)]
+
\log Z(\pi).
\end{align}
Multiplying by \(\tau\) and rearranging gives
\begin{align}
\mathbb E_{\rho}[f_0(u)]
+
\tau\,\mathrm{KL}(\rho\|\pi)
=
-\tau\log Z(\pi)
+
\tau\,\mathrm{KL}(\rho\|\rho^\star_\pi).
\label{eq:variational_identity}
\end{align}
Since \(\mathrm{KL}(\rho\|\rho^\star_\pi)\ge 0\), with equality if and only if
\(\rho=\rho^\star_\pi\) almost everywhere, the minimum value of
\eqref{eq:dist_primal} is
\[
-\tau\log Z(\pi),
\]
and the unique minimizer is
\[
\rho^\star_\pi(u)
=
\frac{
\pi(u)\exp\!\big(-f_0(u)/\tau\big)\mathbf 1_C(u)
}{
Z(\pi)
}.
\]
This proves \eqref{eq:rho_star}.

\subsection{Experiment Setup and Details.}\label{appendix:implement}
\subsubsection{Implementation Details}
We implemented all experiments in Python using JAX~\cite{jax2018github} for batched trajectory rollout and vectorized cost evaluation. For the LQR problem, we use initial state $(x_0=(2.5,0))$, target state $(x^\star=(0,0))$. 
For GD-MPPI, we used 1000 antithetic Gaussian samples per iteration, initialized the control mean to zero, used diagonal variance (across horizon), and ran 20,000 inner iterations. Feasible samples were selected by rejection using the state constraints. As a baseline, we implemented a finite-difference method with perturbation standard deviation $(10^{-3})$, step size $(10^{-3})$, and projection of the updated control sequence back onto the feasible LQR constraint set.
The optimal reference value $(f(x^\star))$ was computed using CVXPY, and convergence was reported as the optimality gap $(f(x)-f(x^\star))$.


\subsubsection{LQR Details}
In this section, we present the derivation of the LQR problem as a QP, similar to the formulation discussed in \cite{yi2024covo}.
We consider a finite-horizon LQR problem with both control and state constraints. The system dynamics are
  \begin{equation}
      x_{t+1} = A x_t + B u_t, \qquad t = 0,\ldots,H-1, \notag
  \end{equation}
  where $x_t \in \mathbb{R}^n$ and $u_t \in \mathbb{R}^m$. 
  The constrained LQR problem is
  \begin{equation}
      \min_{\{u_t\}_{t=0}^{H-1}}
      \frac{1}{2}\sum_{t=0}^{H-1} x_t^\top Q x_t
      +
      \frac{1}{2}\sum_{t=0}^{H-1} u_t^\top R u_t, \notag
  \end{equation}
  subject to
  \begin{align}
      x_{t+1} &= A x_t + B u_t, \notag \\
      u_{\min} &\leq u_t \leq u_{\max}, \notag\\
      x_{\min} &\leq x_t \leq x_{\max}, \notag
  \end{align}
  for $t = 0,\ldots,H-1$, with state constraints applied over the relevant
  trajectory indices.
  To convert this problem into a quadratic program over the control sequence, we
  stack the controls and states as
  \begin{align}
      u =
      \begin{bmatrix}
          u_0^\top & u_1^\top & \cdots & u_{H-1}^\top
      \end{bmatrix}^\top, \notag \\
      x =
      \begin{bmatrix}
          x_0^\top & x_1^\top & \cdots & x_{H-1}^\top
      \end{bmatrix}^\top. \notag
  \end{align}
  Using the linear dynamics, the stacked trajectory can be written as
  \begin{equation}
      x = M u + b,\notag
  \end{equation}
  where $M$ captures the effect of the control sequence on the state trajectory and
  $b$ is the free response from the initial state. For each time step,
  \begin{equation}
      x_t = A^t x_0 + \sum_{j=0}^{t-1} A^{t-j-1} B u_j, \notag
  \end{equation}
  and let
  \begin{equation}
      \bar{Q} = \mathrm{blkdiag}(Q,\ldots,Q),
      \qquad
      \bar{R} = \mathrm{blkdiag}(R,\ldots,R).\notag
  \end{equation}
  Substituting $x = Mu+b$ into the LQR cost gives
  \begin{equation}
      J(u)
      =
      \frac{1}{2}(Mu+b)^\top \bar{Q}(Mu+b)
      +
      \frac{1}{2}u^\top \bar{R}u . \notag
  \end{equation}
  After expanding and removing the constant term independent of $u$, the objective
  becomes
  \begin{equation}
      f(u)
      =
      \frac{1}{2}u^\top Q_{\mathrm{qp}}u + c^\top u, \notag
  \end{equation}
  where
  \begin{equation}
      Q_{\mathrm{qp}} = M^\top \bar{Q}M + \bar{R},
      \qquad
      c = M^\top \bar{Q}b. \notag
  \end{equation}

  The control box constraints remain
  \begin{equation}
      u_{\min} \leq u \leq u_{\max}. \notag
  \end{equation}
  The state constraints are converted into linear constraints on the control
  sequence using $x = Mu+b$:
  \begin{equation}
      x_{\min} \leq Mu+b \leq x_{\max}. \notag
  \end{equation}
  Equivalently,
  \begin{equation}
      x_{\min} - b \leq Mu \leq x_{\max} - b. \notag
  \end{equation}

  Thus, the LQR problem with control and state constraints is written as the
  convex quadratic program
  \begin{equation}
      \begin{aligned}
      \min_{u} \quad
          & \frac{1}{2}u^\top Q_{\mathrm{qp}}u + c^\top u \\
      \mathrm{s.t.} \quad
          & u_{\min} \leq u \leq u_{\max}, \\
          & x_{\min} - b \leq Mu \leq x_{\max} - b .
      \end{aligned}
  \end{equation}

\subsubsection{Derivation of the $L_\Sigma$ for LQR}
In the case of LQR, the objective \(f_0\) is quadratic. Since \(\pi_\theta\) is also Gaussian, it follows that
\[
\rho_\theta \propto \pi_\theta \exp(-f_0 / \tau)
\]
is Gaussian as well. This observation enables us to compute the covariance term needed to evaluate \(\nabla^2 F\), and consequently its norm, i.e., \(L_\Sigma\). In particular, we have
\[
\mathrm{Cov}_{\rho_\theta}(u)
=
\left(\Sigma^{-1} + Q/\tau\right)^{-1}.
\]
Note that the $\mathrm{Cov}$ is independent of $\theta$.
For the scalar variance case \(\Sigma = \sigma^2 I\), according to \eqref{eq:L_sigma_def}, this further simplifies to
\begin{align}
    L_\Sigma
    =
    \left\|
    I - \frac{1}{\sigma^2}
    \left(I/\sigma^2 + Q/\tau\right)^{-1}
    \right\|.
\end{align}

Since the matrix inside the norm is positive semidefinite, \(L_\Sigma\) is equal to its largest eigenvalue. Therefore,
\[
L_\Sigma
=
1 - \frac{\tau}{\tau + \sigma^2 \lambda_{\max}(Q)}.
\]
This shows that increasing \(\sigma^2\) also increases \(L_\Sigma\), and vice versa, and that $\tau$ has the reversed effect on the Lipschitz constant.

\subsection{Proofs}
\label{appendix:proof}

\textbf{Proof of \Cref{lem:gradient-representations}:}
Differentiating $Z(\theta) = \mathbb{E}_{\pi_\theta}[e^{-f_0(u)/\tau}]$ with respect to $\theta$, using the identity
$\nabla_\theta \pi_\theta(u)
=
\pi_\theta(u)\nabla_\theta \log \pi_\theta(u)$, and dividing both sides by $Z(\theta)$ yields
\begin{align}
\nabla_\theta \log Z(\theta)
&=
\int_C \nabla_\theta \log \pi_\theta(u)\,
\frac{\pi_\theta(u)e^{-f_0(u)/\tau}}{Z(\theta)}\,du
\nonumber\\
&=
\mathbb{E}_{\rho_\theta}\!\left[\nabla_\theta \log \pi_\theta(u)\right]. \notag
\end{align}
 Since $F(\theta)=-\tau \log Z(\theta)$, we obtain \eqref{eq:grad_F_theta} and get \eqref{eq:grad_F_theta_2} using \eqref{eq:weight_sample}. 
%
Next, we compute the Hessian of $F(\theta)$. Differentiating 
\eqref{eq:grad_F_theta_2} with respect to $\theta$ gives
\begin{align}
\nabla_\theta^2 F(\theta)
&=
-\tau \nabla_\theta  \frac{\mathbb{E}_{\pi_\theta}\!\left[w(u)\nabla_\theta \log \pi_\theta(u)\right]}
{\mathbb{E}_{\pi_\theta}\!\left[w(u)\right]}.
\label{eq:hessian_step1}
\end{align}
Since $\rho_\theta(u)$ depends on $\theta$
through $\pi_\theta(u)$ and $Z(\theta)$, we differentiate the expectation
using the quotient rule. Define
$
A(\theta)=\mathbb{E}_{\pi_\theta}\!\left[w(u)\nabla_\theta \log \pi_\theta(u)\right],
\ 
B(\theta)=\mathbb{E}_{\pi_\theta}[w(u)].
$
From \eqref{eq:grad_F_theta_2} we have $\nabla_\theta F(\theta)=-\tau \frac{A(\theta)}{B(\theta)}$.
Differentiating this expression yields
\begin{align}
\nabla_\theta^2 F(\theta)
&=
-\tau
\left(
\frac{\nabla_\theta A(\theta)}{B(\theta)}
-
\frac{A(\theta)\nabla_\theta B(\theta)^\top}{B(\theta)^2}
\right).
\label{eq:hessian_step2}
\end{align}
We first compute $\nabla_\theta B(\theta)$. Using the score identity,
\begin{align*}
\nabla_\theta B(\theta) \! = \!
\nabla_\theta \mathbb{E}_{\pi_\theta}[w(u)] 
\! = \!
\mathbb{E}_{\pi_\theta}\!\left[w(u)\nabla_\theta \log \pi_\theta(u)\right]
=
A(\theta).
\end{align*}
Next, differentiating $A(\theta)$ and again using differentiation of
expectations under $\pi_\theta$ gives
\begin{align}
\nabla_\theta A(\theta)
\! = \!
\mathbb{E}\!\left[
w(u)\Big(
\nabla_\theta^2 \log \pi_\theta(u)
\! + \!
\nabla_\theta \log \pi_\theta(u)\nabla_\theta \log \pi_\theta(u)^\top
\Big)
\right]. \notag
\end{align} 
Substituting these expressions into \eqref{eq:hessian_step2} and writing
the result in terms of the distribution $\rho_\theta$ yields \eqref{eq:hessian_general}.

\textbf{Proof of Lemma \ref{lemma:LP_smooth_from_hessian}:}
By Taylor's theorem with integral remainder,
\begin{align}
F(\theta+\Delta\theta)
&= F(\theta)+ \nabla F(\theta)^\top \Delta\theta \notag \\
&+ \int_0^1 (1-t)\,
\underbrace{\Delta\theta^\top \nabla^2 F(\theta+t\Delta\theta)\Delta\theta}_{M}\,dt . \notag
\end{align}
For any $t\in[0,1]$ and $P \succ 0$,
\begin{align}
M &= \Delta\theta^\top P^{-1/2} P^{1/2} \nabla^2 F(\theta+t\Delta\theta) P^{1/2} P^{-1/2}\Delta\theta \notag \\
&\le
\|P^{1/2}\nabla^2 F(\theta+t\Delta\theta)P^{1/2}\|
\,\Delta\theta^\top P^{-1}\Delta\theta \notag \\
&\le
L_P\,\Delta\theta^\top P^{-1}\Delta\theta .\notag
\end{align}
Substituting this bound and using $\int_0^1(1-t)dt=\tfrac12$ yields \eqref{eq:LP_smooth}.

\textbf{Proof of \Cref{thm:gaussian_stepsize_bounded}:}
Define the 1D random variable $Y=x^\top \Sigma^{-1/2} U$, where $U \sim \rho_{\mu}$ and $x \in \mathbb{R}^{n}$ is a unit vector. Since \(\rho_\mu\) is supported on \(C\), $Y$ has a support interval of length at most $D_{\Sigma^{-1}}$. Indeed, 
\begin{align}
\sup_{u,v\in C} |x^\top \Sigma^{-1/2} u - x^\top \Sigma^{-1/2} v|
&\leq
\sup_{u,v\in C} \|x\|\,\|\Sigma^{-1/2}(u-v)\| \notag \\
&\le D_{\Sigma^{-1}}, \notag
\end{align}
since \(\|x\|=1\). Now, among all scalar distributions supported on an interval of length \(D_{\Sigma^{-1}}\), the largest possible variance is \(D_{\Sigma^{-1}}^2/4\), attained by a Bernoulli distribution placing equal mass at the two endpoints. Therefore, $\operatorname{Var}_{\rho_\mu}(Y)\le \frac{D_{\Sigma^{-1}}^2}{4}$. 
Since
\[
x^\top \Sigma^{-1/2} \operatorname{Cov}_{\rho_\mu}(U) \Sigma^{-1/2} x \!
= \!
\operatorname{Var}_{\rho_\mu}(x^\top \Sigma^{-1/2} U) \!
= \!
\operatorname{Var}_{\rho_\mu}(Y),
\]
it follows that $x^\top  \Sigma^{-1/2} \operatorname{Cov}_{\rho_\mu}(U) \Sigma^{-1/2} x \le \frac{D_{\Sigma^{-1}}^2}{4}$ for every unit vector $x$. Hence, $ \Sigma^{-1/2} \operatorname{Cov}_{\rho_\mu}(U) \Sigma^{-1/2}\preceq \frac{D_{\Sigma^{-1}}^2}{4}I$. 
Then, we obtain
\[
\left\|
I-\Sigma^{-1/2}\mathrm{Cov}_{\rho_\mu}(u)\Sigma^{-1/2}
\right\|
\le
\max \{1,\frac{D_{\Sigma^{-1}}^2}{4}-1\}.
\]
Taking the supremum over \(\mu\) proves \eqref{eq:LP_gaussian_bounded}. The step-size condition then follows directly from Theorem~\ref{thm:pgd_general}.

\fi 
\end{document}